\documentclass{llncs}
\usepackage{indentfirst}
\usepackage{graphicx}
\usepackage{amsmath}
\usepackage{chngcntr}
\spnewtheorem{theorem1}{ }{\bfseries}{\itshape}
\counterwithin{theorem1}{section}
\spnewtheorem{lem}[theorem1]{Lemma}{\bfseries}{\itshape}
\spnewtheorem{thm}[theorem1]{Theorem}{\bfseries}{\itshape}
\spnewtheorem{pro}[theorem1]{Proposition}{\bfseries}{\itshape}
\spnewtheorem{cor}[theorem1]{Corollary}{\bfseries}{\itshape}
\spnewtheorem{defi}[theorem1]{Definition}{\bfseries}{\itshape}
\spnewtheorem{rem}[theorem1]{Remark}{\bfseries}{\itshape}
\spnewtheorem{exa}[theorem1]{Example}{\bfseries}{\itshape}

\begin{document}

\title{SEPARATION IN SIMPLY LINKED NEIGHBOURLY 4-POLYTOPES}
\titlerunning{Hamiltonian Mechanics}  
%
\author{T.BISZTRICZKY}
\institute{  }

\maketitle              
\thispagestyle{plain}
\begin{abstract}
The Separation Problem asks for the minimum number $s(O,K)$ 
of hyperplanes required to strictly separate any interior point $O$ 
of a convex body $K$ from all faces of $K$. The Conjecture is $s(O,K)$
$\le 2^d$ in $\bbbr^d$, and we verify this for the class of 
simply linked neighbourly 4-polytopes.
\textbf{keywords:} Gohberg-Markus-Hadwiger covering number, neighbourly polytope, separation problem
\end{abstract}
\section{INTRODUCTION}
We recall that the Separation Problem is the polar version of the Gohberg-Markus-Hadwiger Covering Problem for convex bodies, and refer to [2], [6] and [9] for an overview of the topic.

For convex d-polytopes $P$, the Conjecture has been verified in the case that $P$ is cyclic or a type of neighbourly 4-polytope( totally-sewn or with at most ten vertices). We refer to [3] for an overview of these results.

In the following, we assume that $P$ is a neighbourly 4-dimensional polytope in $\bbbr^4$. Then $P$ is convex and any two distinct vertices determine an edge of $P$. We refer to [8] and [12] for the basic geometric and combinatorial properties of $P$.

With formal definitions to follow; we note only that cyclic polytopes are neighbourly and totally-sewn, and that totally-sewn $P$ are linked. Thus, we verify the Conjecture for a new class of $P$.

As for organization: Section 2 contains definitions and conventions. In Section 3, we examine the inner structure of $P$. In Section 4, we determine some separation properties of P. We introduce simply linked $P$ and present our separation results in Section 5 and 6.

\section{DEFINITIONS}
Let $Y$ be a set of points in $\bbbr^d$. Then $conv\  Y$ and \textit{aff} $Y$ denote, respectively, the convex hull and the affine hull of $Y$. For sets $Y_1,Y_2,\cdots,Y_k$, let
\begin{eqnarray*}
[Y_1,Y_2,\cdots,Y_k]= conv (Y_1\cup Y_2 \cup ... \cup Y_k)
\end{eqnarray*}
and $\langle Y_1,Y_2,\cdots,Y_k \rangle =$ \textit{aff} $(Y_1\cup Y_2 \cup \cdots \cup Y_k)$. For a point $y$, let $[y]=[\lbrace y\rbrace]$ and $\langle y \rangle =\langle \lbrace y \rbrace \rangle$.

Let $Q\in \bbbr^d $denote a (convex) d-polytope with $\mathcal{V}(Q),\mathcal{E}(Q)$ and $\mathcal{F}(Q)$ denoting, respectively, its sets of vertices, edges and facets. For $x\in \mathcal{V}(Q),\  Q\slash _x$ denotes the vertex figure of Q at x. For $E=[x,y]\in \mathcal{E}(Q),\  Q\slash _E$ denotes the quotient polytope $(Q\slash _y)\slash _x$. We note that $Q\slash_E$ is a $(d-2)$-polytope.

Let $d=4$. As a simplification, we assume always that $Q\slash _x$ is contained in a hyperplane $H\subset \bbbr^4$ that strictly separates $x$ from each $y\in \mathcal{V}(Q) \slash \lbrace x \rbrace$, and denote $H \cap [x,y]=H \cap \langle x,y \rangle$ also by $y$. Then of importance here are the following:
\begin{theorem1}
For $y_i\in \mathcal{V}(Q) \slash \lbrace x \rbrace $; a plane $\langle y_1,y_2,y_3 \rangle $ separates $y_4$ and $y_5$ in $\langle Q \slash _x \rangle $ if, and only if, the hyperplane $\langle x,y_1,y_2,y_3 \rangle$ separates $y_4$ and $y_5$ in $\bbbr^4$, and
\end{theorem1}

\begin{theorem1}
For $y_i \in \mathcal{V}(Q) \slash E$; a line $\langle z_1,z_2 \rangle$ separates $z_3$ and $z_4$ in $\langle  Q \slash _E \rangle $ if, and only if the hyperplane$\langle E,z_1,z_2 \rangle$ separates $z_3$ and $z_4$ in $\bbbr^4$.
\end{theorem1}

Let $S\subset \bbbr^3$ be a 3-polytope with $s\geq 4$ vertices. Then $S$ is \textit{stacked} if  either $s=4$ or $S$ is the convex hull of a stacked 3-polytope with $s-1$ vertices and a point in $\bbbr^3$ that is beyond exactly one facet of $S$.

Let $S$ be stacked, $\lbrace x,y,z \rbrace \subset \mathcal{V}(S)$ and $C=[x,y,z]$ be a triangle. We say that C is a \textit{cut} of $S$ if $\mathcal{E}(C) \subset \mathcal{E}(S)$ but $C\notin \mathcal{F}(S)$. All the cuts of $S$ decompose $S$ into \textit{components}, each of which is a 3-simplex. We note that $| \mathcal{V}(S)|=s$ yields that $S$ has $s-4$ cuts and $s-3$ components.

Let $\mathcal{N}_m$ denote the family of combinatorially distinct neighbourly 4-polytopes with $m\geq 5$ vertices, $P\in \mathcal{N}_{m+1},x\in\mathcal{V}(P)$ and $Q=[\mathcal{V}(P)\backslash \lbrace x\rbrace]$. We note that $Q\in \mathcal{N}_m$.

The relevance of stacked 3-polytopes  here is the following result in [1]:

\begin{theorem1}
$P\slash _x$ is a stacked 3-polytope with $m$ vertices; furthermore, $[y_1,y_2,y_3,y_4]$ is a component of $P\slash _x$ if, and only if, $[y_1,y_2,y_3,y_4]\in \mathcal{F}(Q)\backslash \mathcal{F}(P)$. Hence, $x$ is beyond exactly $m-3$ facets of $Q$.
\end{theorem1}

Next, let $E=[x,y]\in \mathcal{E}(P)$. Then $E$ is a \textit{universal edge} of $P$ if $[E,z]$ is a 2-face of $P$ for each $z\in \mathcal{V}(P)\slash \lbrace x,y \rbrace$. Let $\mathcal{U}(P)$ denote the \textit{set of universal edges} of $P$. We observe from [12] and [13] that 

\begin{theorem1}
$E=[x,y]\in \mathcal{U}(P)$ if, and only if, x and y lies on the same side of every hyperplane determined by the vertices of $P$. From the same sources; if $| \mathcal{V}(P) \geq 7$ then any vertex of $P$ is on at most two members of $\mathcal{U}(P)$, and $|\mathcal{U}(P)|\leq | \mathcal{V}(P) |$.
\end{theorem1}

We recall that a \textit{cyclic 4-polytope} $C_m$ with $m$ vertices is combinatorially equivalent to the convex hull of $m$ points on the moment curve in $\bbbr^4$. From [7], [8] and [12], we note that $C_m\in \mathcal{N}_m$, $ \mathcal{N}_6=\lbrace C_6 \rbrace$, $ |\mathcal{U}(C_6)|=9$, $\mathcal{N}_7=\lbrace C_7 \rbrace$, $ |\mathcal{U}(C_m)|=m$ for $m\geq 7$, and any 4-subpolytope of $C_m$ is again cyclic. For $m\geq 6$, there is a natural ordering (Gale's Evenness Condition) of $\mathcal{V}(P_m)$ that corresponds to the order of appearance of equivalent points on the moment curve.

Let $m \geq 8$. Most of our knowledge about members of $\mathcal{N}_m$ is based upon various \textit{construction techniques}: given $Q\in \mathcal{N}_{m-1}$, find a point $\bar{x}\in \bbbr^4 \slash Q$ such that $\bar{Q}=[Q,\bar{x}]\in \mathcal{N}_m$. It is noteworhy that, at present, known constructions such as Shemer Sewing, Extended Sewing and Gale Sewing(cf. [12], [10] and [11]) yield that $\mathcal{U}(\bar{Q})\backslash \mathcal{U}(Q)\neq \emptyset$. We introduce a class of polytopes to reflect this fact.

Let $n\geq 7$ and $P_n\in \mathcal{N}_n$. We say that $P_n$ is \textit{linked} if for $m=n-1,\cdots,6$, there is a $P_m\in \mathcal{N}_m$ with the property that
\begin{eqnarray*}
P_{m+1} \supset P_{m} \  and\   \mathcal{U}(P_{m+1})\backslash \mathcal{U}(P_m)\neq \emptyset
\end{eqnarray*}
We say that $P_n$ is \textit{linked under the (vertex) array} $x_n>x_{n-1}>\cdots >x_1$ if for $m=n-1,\cdots,6$,
\begin{eqnarray*}
P_m=[x_m,x_{m-1},\cdots,x_1] \  and\  \mathcal{U}(P_{m+1})\backslash \mathcal{U}(P_m)\neq \emptyset 
\end{eqnarray*}
For $x_t\in \lbrace x_7,\cdots,x_n \rbrace$ and $x_r\in \lbrace x_1,\cdots,x_{t-1} \rbrace$, we say that $x_t$ \textit{is linked to} $x_r\  (x_t \rightarrow x_r)$ if $[x_t,x_r]\in \mathcal{U}(P_t)$ and $[x_t,x_j]\notin \mathcal{U}(P_t)$ for $j > max\lbrace 6,r \rbrace$.

By way of clarification for requiring that $t\geq 7$; we note that
\begin{theorem1}
$P_6$ is cyclic and there are disjoint three element subsets $Y$ and $Z$ of $\mathcal{V}(P_6)$ such that $\mathcal{U}(P_6)=\lbrace[y,z] | y\in Y \  and\   z\in Z\rbrace$. Thus, there is no meaningful labeling of a greatest or a least vertex of $P_6$
\end{theorem1}
\section{THE INNER STRUCTURE OF $P$}

Let $v\in \mathcal{V}(P),\  Q\subset P,\  v\notin Q ,\  Q\in \mathcal{N}_m$ and $R=[Q,v]$. We recall that \\$R^*=R\slash _v$ is a stacked 3-polytope and that $y\in \mathcal{V}(Q)$ denotes also $\lbrace y^* \rbrace=\langle v,y \rangle \cap \langle R^* \rangle$. We describe $R^*$.

Let
\begin{equation*}
Y_a=\lbrace y_1,y_2,\cdots,y_a\rbrace,\  z\in \mathcal{V}(Q) \backslash Y_a,
\end{equation*}

\begin{equation*}
    Z_t=\lbrace z| \langle v,y_1,y_t,y_{t+1}\rangle strictly\  separates\  y_2\  and\  z\rbrace
\end{equation*}
and 
\begin{equation*}
    Z'_t=\lbrace z| \langle v,y_2,y_t,y_{t+1}\rangle strictly\  separates\  y_1\  and\  z\rbrace
\end{equation*}
From 2.1, we have that
\begin{itemize}
\item $Z_t\neq \emptyset (Z'_t \neq \emptyset)$  if  and  only  if, $ [y_1,y_t,y_{t+1}]\  ([y_2,y_t,y_{t+1}])$  is  a  cut  of $ R^*$.
\end{itemize}
Hence, we have a generic description of $R^*$; cf. the Schlegel diagram in Figure 1.

Next, we observe from 2.3 and 2.1 that
\begin{itemize}
\item $\langle y_1,y_2,y_t,y_{t+1} \rangle$  separates $ v$  and $Q$  for $ t=3,\cdots,a-1$, and
\item  $\langle v,y_1,y_2,y_{t} \rangle$  separates $ Z_r\cup Z'_r(r<t)$  and $ Z_s\cap Z'_s(s\geq t)$  for $ t=4,\cdots,a-1.$
\end{itemize} 

From 2.2, we depict these separation properties with respect to $Q\slash _{[y_1,y_2]}$ and $R\slash _{[y_1,y_2]}$ in Figure 2.\\  \\ \textbf{REMARKS} Let $R^*=R\slash_v$ be labeled as above.

\begin{figure}[htbp]                                                         
\centering\includegraphics[width=4in]{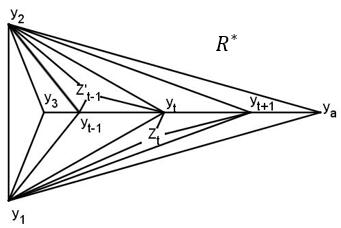}       
\caption{ }\label{fig:1}                                                     
\end{figure}                                                                     
\begin{figure}[htbp]                                                         
\centering\includegraphics[width=4in]{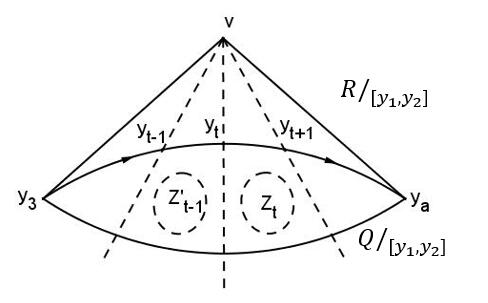}       
\caption{ }\label{fig:2}                                                     
\end{figure}                                                                     

\begin{theorem1}
If $a=m$ then $\mathcal{F}(Q)\backslash \mathcal{F}(R)=\lbrace [y_1,y_2,y_t,y_{t+1}]|\  t=3,\cdots,a-1\rbrace,\  [y_1,y_2]\in \mathcal{U}(Q)$ and $\lbrace [v,y_1],[v,y_2] \rbrace \subset \mathcal{U}(R).$
\end{theorem1}

There is such a labeling of $R^*$ if $R$ is cyclic or if $R$ is constructed by a Shemer sewing of $v$ onto $Q$.

\begin{theorem1}
If $w\in \mathcal{V}(P)\backslash \mathcal{V}(Q)$ and $[w,v]\in \mathcal{U}([Q,v,w])$ then $\mathcal{F}(Q)\backslash \mathcal{F}([Q,w])=\mathcal{F}(Q)\backslash \mathcal{F}(R)$  by 2.4.
\end{theorem1}

\begin{theorem1}
Let $3<t<a$. If $\langle v,y_1,y_2,y_t\rangle$ strictly separates vertices $p_t$ and $s_t$ of $Q$ then $[p_t,s_t]$ is not an edge of $R^*$, and $[v,p_t,s_t]$ is not a face of $R$.
\end{theorem1}

We note from Figure 2 that under the hypotheses of 3.3, each hyperplane through $\langle y_1,y_2,y_t\rangle$ strictly separates some two of $v$, $p_t$ and $s_t$. Thus, the following is the more general result; cf. [5].

\begin{theorem1}
If $\lbrace x_a,x_b,x_c,x_e,x_f,x_g\rbrace$ is a set of six vertices of $P$ and each hyperplane of $\bbbr^4$ though $\lbrace x_a,x_b,x_c\rbrace$ strictly separates two of $x_e,x_f$ and $x_g$, then $[x_a,x_b,x_c]$ and $[x_e,x_f,x_g]$ are not faces of $P$.
\end{theorem1}


\section{GENERIC SEPARATION PROPERTIES OF P}
Let $P\in \mathcal{N}_m,m\geq 6$, and $O$ be an interior point of $P$. We determine hyperplanes $H\in \bbbr^4$ that strictly separate $O$ from facets of $P$. As a simplification, we determine $H$ that do not contain $O$. We consider first $F\in \mathcal{F}(P)$ that either are contained in a subpolytope $Q$ such that $O\notin int\  Q$ or have a common vertex $w$.

\begin{lemma}
( cf. [4] ) \textit{ Let $O\in bd(Q)$. Then $O$ is strictly separated from any $F\in \mathcal{F}(P)\cap \mathcal{F}(Q)$ by one of at most three hyperplanes.}
\end{lemma}

\begin{lemma}
\textit{Let $w\in \mathcal{V}(P), R\in \mathcal{N}_{m-1},P=[R,w]$ and $F\in \mathcal{F}(P)$ such that $w\in F$. Then $O$ is strictly separated from any such $F$ by one of at most four(six) hyperplanes in case $O$ is (is not) an interior point of $R$.}
\end{lemma}

\begin{proof}
Since $P^*=P\slash _w$ is stacked and $O\in int\  P$, it follows that $O^*\in \langle w,O\rangle \cap P^*$ is in a component $A^*=[x^*,y^*,z^*,v^*]$ of $P^*$. If $O^*\in relint\  A^*$ then $O$ is separated from $F$ by one of $\langle w,x,y,z \rangle ,\  \langle w,x,y,v \rangle,\  \langle w,x,z,v \rangle$ and $\langle w,y,z,v \rangle$.

Let $O^*\in B^*=[x^*,y^*,z^*]$, say. Then $B^*$ is a cut of $P^*,\  O\in \langle w,x,y,z \rangle$ and there are subpolytopes $P'$ and $P''$ of $P$ such that $P' \cap  P'' =[w,x,y,z],[P',P'']=P$ and (since $w\in F$) either $F\subset P'$ or $F\subset P''$.

We recall from $2.3$ that $[x,y,z,v]\in \mathcal{F}(R)\backslash \mathcal{F}(P)$. If $O\in int \  R$ then it is clear that $O\notin [w,x,y,z]$; that is, $O\notin P'\cup P''$ and $O$ is separated from $F$ by one of two hyperplanes. If $O\in [w,x,y,z]$ then $O\in bd(P')\cap bd(P'')$ and we apply LEMMA A. \qed
\end{proof}
\textbf{REMARKS} Let $Q$ be a subpolytope of $P$ such that $O\notin int\  Q$.

\begin{theorem1}
If $Q\in \mathcal{N}_{m-1}$ then $O$ is strictly separated from any $F\in \mathcal{F}(P)$ by one of at most nine( three from LEMMA A, six from LEMMA B) hyperplanes.
\end{theorem1}

\begin{theorem1}
If $Q\in \mathcal{N}_{m-3}$ then $O$ is strictly separated from any $F\in \mathcal{F}(P)$ by one of at most fifteen hyperplanes.
\end{theorem1}

For $4.2$, we apply $LEMMA\  B$ under the assumption that $O$ is an interior point of any $Q'\in \mathcal{N}_{m-1}$ such that $Q'\subset P$

\section{SIMPLY LINKED P}

Let $n\geq 7$ and $P=P_n\in \mathcal{N}_n$ be linked under the array $x_n>x_{n-1}>\cdots >x_1$.

Let $\mathcal{W}=\lbrace w_s,w_{s-1},\cdots,w_1\rbrace$ be an s element subset of $\mathcal{V}(P)$ with the induced array $ w_s>w_{s-1}>\cdots>w_1$ in the case $s>1$. Then $\mathcal{W}$ is a \textit{chain} if either $s=1$ or 

\begin{eqnarray*}
w_s\rightarrow  w_{s-1} \rightarrow \cdots \rightarrow w_1.
\end{eqnarray*}

For $x_k\in \mathcal{V}(P)$, let $\mathcal{V}^k$ denote the union of all chains of $P$ with $x_k$ as the least vertex.

Finally, we say that $P_n$ is \textit{simply linked} if for $k=7,\cdots ,n:$
\begin{itemize}
\item $\mathcal{V}^k$ is a chain, and
\item for disjoint chains $\mathcal{V}^c$ and $\mathcal{V}^d$, there are $x_i \neq x_j$ in $\mathcal{V}(P_6)$ such that $x_c\rightarrow x_i,x_d\rightarrow x_j$ and $[x_i,x_j]\notin \mathcal{U}(P_6)$.
\end{itemize}

Henceforth, we assume that $P_n$ is simply linked. Then it follows from $2.5$ that $\lbrace x_7,\cdots ,x_n \rbrace$ is the union of at most three pairwise disjoint maximal chains.

\begin{lemma}
Let $6 \leq m<n, x_m<x_t,x_k<x_t$ and $x_t \notin \mathcal{V}^k.$\\ \textbf{C.1} $H\cap [\mathcal{V}^t]=\emptyset$ for any hyperplane $H$ spanned from $\lbrace x_1,\cdots ,x_m \rbrace$.\\ \textbf{C.2} Let $H_h=\langle x_a,x_b,x_c,x_h \rangle$ be a hyperplane with $\lbrace x_a,x_b,x_c\rbrace \subset \lbrace x_1,\cdots ,x_m\rbrace$ and $H_h \cap \mathcal{V}^k=\lbrace x_h \rbrace$. Then $H_h\cap [\mathcal{V}^t]=\emptyset$.\\ \textbf{C.3} Let $x_t\rightarrow x_j,x_j\notin \lbrace x_a,x_b,x_c\rbrace$ and $H_h$ be defined as above. Then $H_h\cap [\mathcal{V}^j]=\emptyset$.
\end{lemma}
\begin{proof}
Since $P$ is simplicial, it follows from $H\cap [\mathcal{V}^t]\neq  \emptyset$ that $H$ strictly separates some $x_v$ and $x_u$ in the chain $\mathcal{V}^t$ such that $x_v \rightarrow x_u$. Then $[x_v,x_u]\in \mathcal{U}(P_v)$ and $P_m\subset P_t\subset P_v$ yield a contradiction by $2.4$.

As above, $H_h\cap [\mathcal{V}^t]\neq \emptyset$ implies that $H_h$ strictly separates some $x_s$ and $x_q$ in $\mathcal{V}^t$ such that $x_s\rightarrow x_q$. Thus, C.1 yields that $x_s<x_h$ and $x_t\in P_s\subset P_h$. From $x_h \in \mathcal{V}^k$ and $x_k<x_t<x_h$, there is an $x_g \in \mathcal{V}^k$ such that $x_h\rightarrow x_g$. Then $[x_g,x_h]\in \mathcal{U}(P_h),x_m<x_t<x_h, \mathcal{V}^k\cap \lbrace x_a,x_b,x_c \rbrace = \emptyset$ and $2.4$ yield that in the pencil of hyperplanes containing $\langle x_a,x_b,x_c\rangle:$
\begin{eqnarray*}
\langle x_a,x_b,x_c,x_s\rangle \cap [x_g,x_h]=\emptyset =\langle x_a,x_b,x_c,x_q\rangle \cap [x_g,x_h].
\end{eqnarray*}
 Hence, $\langle x_a,x_b,x_c,x_g\rangle$also strictly separates $x_s$ and $x_q$, and $x_t\in P_s\subset P_g \subset P_h.$ It now follows from $x_h\rightarrow x_g\rightarrow \cdots \rightarrow x_k$ that $x_t\in P_s\subset P_k\subset P_h;$ a contradiction.

We note that $\mathcal{V}^j=\lbrace x_j\rbrace \cup \mathcal{V}^t$ and 
that if $H_h \cap [\mathcal{V}^j]\neq \emptyset $ then $H_h$ strictly separates $x_t$ and $x_j$ by C.2, and $x_t<x_h$ by C.1. We now argue on above and obtain a contradiction. \qed
\end{proof}
\textbf{REMARKS} We recall that $P_m=[x_m,x_{m-1},\cdots,x_1]$ for $m=n,\cdots,6.$ Let $P_5$ denote any 4-subpolytope of $P_6$. In view of $2.5$,
\begin{theorem1}
there is a labeling of $\mathcal{V}(P_6)$, which we may denote by $x_1,x_2\cdots ,x_6$, such that 
\begin{itemize}
\item $P_6$ satisfies Gale's Evenness Condition with $x_1<x_2<\cdots <x_6, Y=\lbrace x_1,x_3,x_5 \rbrace,Z=\lbrace x_2,x_4,x_6\rbrace$
\item $P_5=[x_1,x_2,\cdots,x_5]$, and
\item any hyperplane through $\langle Y \rangle$ strictly separates two elements of $Z$.
\end{itemize}
\end{theorem1}

We recall that $P=P_n$ is simply linked under $x_n>x_{n-1}>\cdots >x_1$ and $P_m=[x_m,\cdots,x_1]$ for $m\geq 5$. Let $O$ be an interior point of $P,6\leq m\leq n-1$ and $O\in P_m \backslash P_{m-1}$. We note that a vertex of $[x_{m+1},\cdots ,x_n]$ is linked to a vertex of $P_m$.

With $v=x_w>x_m,Q=P_m$ and $R=[Q,v]$, we label $Q$ and $R^*=R\slash _v$\\ as in Section 3 so that $x_w\rightarrow y_1$ (hence, each $Z'_t$ is empty) and $\langle x_w,O\rangle \cap [y_1,y_2,y_t,y_{t+1}]\neq \emptyset$ for some $3\leq t\leq a-1$. We let $T=[y_1,y_2,y_t,y_{t+1}]$,
\begin{eqnarray*}
Z_t^-=\lbrace y_3,\cdots , y_{t-1} \rbrace \cup Z_3 \cup \cdots \cup Z_{t-1},\\Z_t^+=\lbrace y_{t+2},\cdots , y_a\rbrace \cup Z_{t+1} \cup \cdots \cup Z_{a-1}
\end{eqnarray*}
and note that \\
\begin{centering}
$T\in \mathcal{F}(Q)\backslash \mathcal{F}(R)$ and $\mathcal{V}(P_m)=Y_a\cup Z_3\cup \cdots \cup Z_{a-1}=\mathcal{V}(T)\cup Z_t^- \cup Z_t\cup Z_t^+$.
\end{centering}
\begin{figure}[htbp]                                                         
\centering\includegraphics[width=4in]{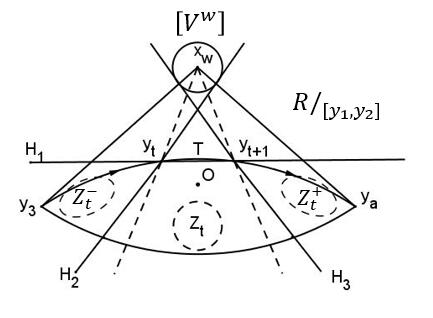}       
\caption{ }\label{fig:4}                                                     
\end{figure}                                                                     

From the Schlegel diagram of $R^*$ on $[y_1,y_2,y_a]$ in Figure 1, we readily obtain diagrams of $R^*$ on 2-faces containing $[y_1,y_t]$ or $[y_2,y_t]$. In Figure 3, 4 and 5, we depict associated polygons $R\slash _{[y_1,y_2]},R\slash _{[y_1,y_t]}$ and $ R\slash _{[y_2,y_t]}$ that include $[\mathcal{V}^w]$(as per 3.2 and C.1) and hyperplanes $H_1,H_2,H_3,H_4$ and $H_5$ that separate $O$ and $[\mathcal{V}^w]$. We note that each of $H_2,H_3,H_4$ and $H_5$ intersects and supports $[\mathcal{V}^w]$. For $i=1,\cdots ,5$, let $H_i^-$ and $H_i^+$ denote the open half-spaces of $\bbbr ^4$ determined by $H_i$ with $\mathcal{V}^w \subset H_i \cup H_i^+$.\\ \\
\textbf{REMARKS} Let $F\in \mathcal{F}(P)$ and assume by 4.2 that $m\leq n-3$. From $\langle x_w,O\rangle\cap T\neq \emptyset$, we have the following:

\begin{theorem1}
O is separated from all F with a common vertex by one of at most four hyperplanes; cf. LEMMA B.
\end{theorem1}
\begin{theorem1}
$O\in [x_m,P_{m-1}]\backslash P_{m-1}$ is separated from any $F\in \mathcal{F}(P_m)$ by one of at most five hyperplanes.
\end{theorem1}
\begin{theorem1}
If $F\cap \mathcal{V}^w \neq \emptyset$ then $F$ intersects at most one of $Z_t^-,Z_t$ and $Z_t^+$; cf. 3.4.
\end{theorem1}
\begin{theorem1}
If $O\notin T$ then $O$ is separated from any $F$ such that $F\cap \mathcal{V}^w\neq \emptyset$ and $\mathcal{V}(F)\subset \mathcal{V}(P_m)\cup \mathcal{V}^w$ by one of $H_2,H_3,H_4$, and $H_1$ or $H_5$ in  the case $F\cap (Z_t^- \cup Z_t \cup Z_t^+)= \emptyset$.
\end{theorem1}
\begin{theorem1}
Let $x_m\notin T$. Then $O$ is separated from any $F$ such that $\mathcal{V}(F)\subset \mathcal{V}(P_{m-1})\cup \mathcal{V}^w$ by one hyperplane.
\end{theorem1}

\begin{figure}[htbp]                                                         
\centering\includegraphics[width=4in]{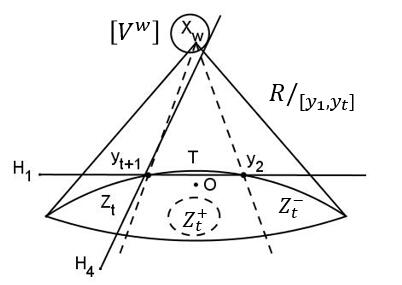}       
\caption{ }\label{fig:5}                                                     
\end{figure}                                                                     

Regarding 5.6; let $\langle x_w,O\rangle \cap T = \lbrace p_1 \rbrace$. Then $\langle x_w,O\rangle \cap bd(P_{m-1})=\lbrace p_1,p_2 \rbrace$ and $\langle x_w,O\rangle \cap bd(P_m)=\lbrace p_1,p_3 \rbrace$ with $p_2 \in [p_1,p_3]$. From $O\in [x_m,P_{m-1}]\backslash P_m$, we obtain that $O\in [p_2,p_3]$. It is now clear that there is an $F'\in \mathcal{F}(P_{m-1})$ such that $p_2\in F'$ and $\langle F' \rangle$ separates $O$ from $F$ with $\mathcal{V}(F)\subset \mathcal{V}(P_{m-1})\cup \mathcal{V}^w$.\qed \  

Since $P$ is simply linked and $m\leq n-3$, we consider the case that $\lbrace x_{m+1},\cdots ,x_n \rbrace$ is the union of mutually disjoint chains $\mathcal{V}^w, \mathcal{V}^s$ and $\mathcal{V}^r$ with $x_w\rightarrow y_1\in P_m,x_s\rightarrow \hat{y}_1 \in P_m$ and $x_r\rightarrow \bar{y}_1=x_m$. Then with labelings analogous to the one for $Q=P_m$ and $R^*[P_m,x_w]\slash _{x_w}$;
\begin{eqnarray*}
\mathcal{V}(P_m)=\hat{Y}_b\cup \hat{Z}_3\cup \cdots \cup \hat{Z}_{b-1}=\lbrace \hat{y}_1,\hat{y}_2,\hat{y}_k,\hat{y}_{k+1}\rbrace \cup \hat{Z}_k^- \cup \hat{Z}_k \cup \hat{Z}_k^+
\end{eqnarray*}
corresponds to $[P_m,x_s]\slash _{x_s}$ and 
\begin{eqnarray*}
 \mathcal{V}(P_m)=\bar{Y}_c\cup \bar{Z}_3\cup \cdots \cup \bar{Z}_{c-1}=\lbrace \bar{y}_1,\bar{y}_2,\bar{y}_i,\bar{y}_{i+1}\rbrace \cup \bar{Z}_i^- \cup \bar{Z}_i \cup \bar{Z}_i^+ 
\end{eqnarray*}
corresponds to $[P_m,x_r]\slash _{x_r}$.

We simplify the notation and let $u_j=\hat{y}_j,U_j=\hat{Z}_j,r_j=\bar{y}_j$ and $V_j=\bar{Z}_j$.

With reference to Figure 3, 4 and 5, we assume that $\lbrace T,L,I\rbrace \subset \mathcal{F}(P_m)$ and that
\begin{itemize}
\item $\langle x_w,O\rangle \cap T\neq \emptyset$ with $T=[y_1,y_2,y_t,y_{t+1}]\notin \mathcal{F}([P_m,x_w])$,
\item $\langle x_s,O\rangle \cap K\neq \emptyset$ with $K=[u_1,u_2,u_k,u_{k+1}]\notin \mathcal{F}([P_m,x_s])$,
\item $\langle x_r,O\rangle \cap I\neq \emptyset$ with $I=[v_1,v_2,v_i,v_{i+1}]\notin \mathcal{F}([P_m,x_r])$,
\item $O$ is separated from $[\mathcal{V}^s]([\mathcal{V}^r])$ by $\hat{H}_1,\cdots ,\hat{H}_5(\bar{H}_1,\cdots,\bar{H}_5)$ and
\item $\mathcal{V}^s \subset \hat{H}_j\cup \hat{H}_j^+(\mathcal{V}^r \subset \bar{H}_j\cup \bar{H}_j^+)$ for $j=1,\cdots , 5$.
\end{itemize}

\begin{figure}[htbp]                                                         
\centering\includegraphics[width=4in]{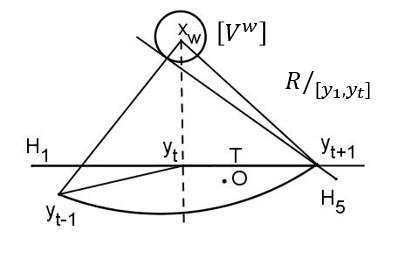}       
\caption{ }\label{fig:1}                                                     
\end{figure}                                                                     
\textbf{REMARK} We refer to Figure 4, and consider any hyperplane $H'$ through $\langle y_1,y_t,y_{t+1}\rangle$ in the case that $x_s \rightarrow u_1=y_2$. Then $[x_s,y_2]\in \mathcal{U}(P_s)$, and it follows from LEMMA C that if $H'\cap Z_t \neq \emptyset$, then $H'\cap [y_2,x_s]=\emptyset,H'\cap [\mathcal{V}^s]=\emptyset$ and $H'$ strictly separates $[\mathcal{V}^s]$ and $[\mathcal{V}^w]$. The following now follows from 3.4:
\begin{theorem1}
If $u_1=y_2$ and $F\cap \mathcal{V}^s\neq \emptyset \neq F\cap \mathcal{V}^w$ then $F\cap Z_t=\emptyset$.
\end{theorem1}

\begin{lemma}
Let $F\in \mathcal{F}(P)$.  Then $F$ intersects at most two of $\mathcal{V}^2,\mathcal{V}^4$ and $\mathcal{V}^6$, and at most two of $\mathcal{V}^w,\mathcal{V}^s$ and $\mathcal{V}^r$.
\end{lemma}
\begin{proof}
The existence of $\mathcal{V}^w,\mathcal{V}^s$ and $\mathcal{V}^r$ imply that $\lbrace x_7,\cdots,x_n\rbrace$ is the union of pairwise disjoint chains $\mathcal{V}^e,\mathcal{V}^f$ and $\mathcal{V}^g$, say. Since $P_6$ is cyclic with $x_1<x_2<\cdots <x_6$, we may assume by 2.5 and 5.1 that $x_e\rightarrow x_2,x_f\rightarrow x_4$ and $x_g\rightarrow x_6$.

From 5.1 and LEMMA C, we obtain that any hyperplane $H$ through $\langle x_1,x_3,x_5 \rangle$ strictly separates two of $\mathcal{V}^2,\mathcal{V}^4$ and $\mathcal{V}^6$. Hence, no face of $P$ intersects each of $\mathcal{V}^2,\mathcal{V}^4$ and $\mathcal{V}^6$ by 3.4. \qed
\end{proof}
\textbf{REMARK} We refer to Figure 3, 4 and 5, and consider a $v\in \mathcal{V}(P)$ with the property that $v\in H_1^-,v\notin H_2^+\cup H_3^+\cup H_4^+ $ and no hyperplane spanned from $v,y_1,y_2,y_t,y_{t+1}$ intersects $[\mathcal{V}^w]$.

Then $\langle x_w,y_1,y_2,y_t\rangle$ separates $v$ and $Z_t^-, \langle x_w,y_1,y_2,y_{t+1}\rangle$ separates $v$ and $Z_t^+$, and $\langle x_w,y_1,y_t,y_{t+1}\rangle$ separates $v$ and $Z_t$. From $[P_m,x_w]\slash _{[x_w,y_1]}$, it now follows that $[x_w,y_1,v]$ is not a 2-face of $[P_m,x_w,v]$ or $[x_w,y_1,y_2,y_t,y_{t+1},v]$. Since the latter polytope is cyclic, we obtain from 5.1 that
\begin{theorem1}
$\langle v,y_2,y_t,y_{t+1}\rangle$ strictly separates $x_w$ and $y_1$.
\end{theorem1}

\begin{lemma}
Let $x_w<x_s,F\in \mathcal{F}(P)$ and $F\cap \mathcal{V}^w \neq \emptyset \neq F\cap \mathcal{V}^s$. Then $O$ is separated from any such F by\\ \textbf{E.1} at most three hyperplanes $(\hat{H}_2,\hat{H}_3,\hat{H}_4)$ in the case $x_w\in \hat{H}_1^-$ and $O\notin bd(K)$, \\ \textbf{E.2} one hyperplane $( H_i,2\leq i\leq 4)$ in the case $u_1\notin T$ and $O\notin bd(T)$,\\ \textbf{E.3} one hyperplane $(H_i,2\leq i \leq 5)$ in the case $x_w\in \hat{H}_i^+,u_1\in T, x_s \in H_1^-$ and $O\notin bd(T)$, and \\ \textbf{E.4} at most two hyperplanes from $H_2,H_3,\hat{H}_2,\hat{H}_3$ in the case $x_w\in \hat
H_1^+,x_s\in H_1^+, O\notin bd(K) \cup bd(T)$ and either $T\neq K$ , or $T=K$ and $F\cap H_1^-\neq \emptyset \neq F\cap \hat{H}_2^-$.
\end{lemma} 
\begin{proof}
We refer to Figure 3, 4 and 5, and the analogous figures with $\mathcal{V}^s,K,U_j=\hat{Z}_j$ and $\hat{H}_j$ for the location of $O$, and note $\mathcal{V}(F)\subset \mathcal{V}(P_m)\cup \mathcal{V}^w\cup \mathcal{V^s}$ by LEMMA D.

\textbf{E.1} Let $x_w\in \hat{H}_1^-$. Then $\mathcal{V}^w\subset \hat{H}_1^-$ by C.1, and either $(\mathcal{V}(F)\cap \hat{H}_1^-) \cap \hat{H}_j^+\neq \emptyset$ for some $j\in \lbrace 2,3,4 \rbrace$ or there is a $v\in \mathcal{V}^w$ such that $v\notin \hat{H}_2^+ \cup \hat{H}_3^+ \cup \hat{H}_4^+$. In case of the former, $O$ is separated from $F$ by $\hat{H}_j$; cf. 5.4. In case of the latter, it follows from $x_w<x_s$ and C.3 that $\langle v,u_2,u_k,u_{k+1}\cap [\mathcal{V}^s \cup \lbrace u_1\rbrace]=\emptyset$; a contradiction of 5.8.

\textbf{E.2} Let $u_1\notin T$. Then $x_s\rightarrow u_1 \in P_m$ yields that $u_1\in Z_t^-\cup Z_t\cup Z_t^+ \subset H_2^+\cup H_3^+ \cup H_4^+$ and $\mathcal{V}^s\subset H_1^-$. Now $x_w<x_s$ and C.1 yield that if $u_1\subset H_j^+$ then $\mathcal{V}^s\subset H_j^+$ and $O$ is separated from $F$ by $H_j$.

\textbf{E.3} Let $x_w\in \hat{H}_1^+$ and $u_1\in T$. Then $u_1\in \lbrace y_2,y_y,y_{t+1} \rbrace, y_1\in \lbrace u_2,u_k,u_{k+1} \rbrace$ and may assume that $u_1=y_2$ and $y_1=u_2$. From 5.7, we obtain that $F\cap Z_t=\emptyset =F\cap U_k$.

Let $x_s\in H_1^-$. Then $\langle T \rangle =H_1\neq \hat{H}_1=\langle K \rangle$ and $T\neq K$; cf. Figure 6 with $x_s\in H_3^+$, say, and $\mathcal{V}^s\subset H_1^-\cap H_3^+$. We consider the hyperplanes through $\langle y_1,y_2,y_{t+1}\rangle$ and obtain from 3.4 that $F\cap Z_t^-=\emptyset$. Hence, $O$ is separated from $F$ by $H_3$ in this case.

\textbf{E.4} Let $x_w\in \hat{H}_1^+,x_s\in H_1^+$ and $O\notin bd(K)\cup bd(T)$. Then $u_1\in T$, we assume that $(y_1,y_2)=(u_2,u_1)$ and note that $F\cap (Z_t\cup U_k)=\emptyset$.

If $T\neq K$ then we may assume also that $\mathcal{V}^s\subset H_1^+\cup H_3^+$ as in Figure 6. If there is an $F'\in \mathcal{F}(P)$ such that $O$ is not separated from $F'$ by $H_3$ then $F'\cap Z_t^+=\emptyset$ by 5.4. From $F'\cap Z_t= \emptyset$, it now follows that $F'\cap (Z_t^-\cup \lbrace y_t\rbrace)\neq \emptyset$.

Let $Z_t^-(s)=\lbrace z\in Z_t^- | \langle y_1,y_2,y_{t+1},z\rangle\  does\  not\  separate\  x_w\  and\  x_s\rbrace$  . We apply C.1 and 3.4, and obtain that $F'\cap Z_t^-\subset Z_t^-(s)$. Thus, either $O$ is separated from $F'$ by $H_1$ (and so $\hat{H}_2$), or there is an $F'$ such that $F' \cap Z_t^-\neq \emptyset$. In the latter case, it is easy to check (cf. Figure 6) that $O$ is separated from any such $F'$ by $H_2$ or $\hat{H}_2$.

Finally, let $T=K$ with $(y_1,y_2,y_t,y_{t+1})=(u_2,u_1,u_k,u_{k+1})$ and, say, $\mathcal{V}^s\subset H_3^+$. We argue now as above that if $O$ is not separated from $F'\in \mathcal{F}(P)$ by $H_3$ or $H_1=\hat{H}_1$ then $O$ is separated from $F'$ by $H_2$ or $\hat{H}_2$. \qed

\end{proof}
\textbf{REMARK} We observe that under the hypotheses of LEMMA E, it follows from LEMMA A that
\begin{theorem1}
If $O\in bd(k)\cup bd(T)$ and $H$ is a separating hyperplane through $O$ then we may replace $H$ by three strictly separating hyperplanes.
\end{theorem1}

\begin{figure}[htbp]                                                         
\centering\includegraphics[width=4in]{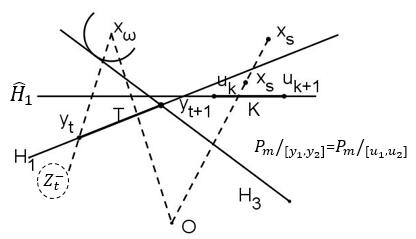}       
\caption{ }\label{fig:1}                                                     
\end{figure}                                                                     

\section{SEPARATION RESULTS}

With \textit{O}$\in$\textit{intP}, let $s(O)$ denote the minimum number of hyperplanes required to strictly separate $O$ from any facet of $P$. We prove that $s(O)\leq 16$ under the assumption that $P=P_n$ is simply linked under the array $x_n>x_{n-1}>\cdots >x_1 ,P_m=[x_m,\cdots,x_2,x_1]$ for $m=n-1,\cdots, 5$ and $\mathcal{V}(P)=\lbrace x_1,x_3,x_5\rbrace \cup \mathcal{V}^2 \cup \mathcal{V}^4 \cup \mathcal{V}^6$.

\begin{theorem1}
We consider first the case of $O\in P_m \backslash P_{m-1}$ for some $6\leq m \leq n$. As noted in Sections 4 and 5, we may assume that $m\leq n-3$ and that $\lbrace x_{m+1},\cdots,x_n\rbrace$ is the union of non-empty chains $\mathcal{V}^r,\mathcal{V}^s$ and $\mathcal{V}^w$ described in Section 5. 
\end{theorem1}
Our arguments are based upon
\begin{itemize}
\item the location of $x_m$ with respect to $T$ and $K$,
\item the order of $x_r$ with respect to $x_w<x_s$, and
\item the location of $O$ with respect to $T,K$ and $I$.
\end{itemize}

For each location of $O$, we present the separation result
\begin{itemize}
\item $\lbrace k\rbrace$: property: rationale
\end{itemize}
to indicate that at most $k$ separating hyperplanes suffice for $F\in \mathcal{F}(P)$ with the indicated property due to the specified reasons. $\lbrace -\rbrace$ indicates that the separating hyperplanes for this case have already counted.
\begin{center}
\textbf{\uppercase\expandafter{\romannumeral1}.} $x_m\notin T\cup K$
\end{center}

Then $T\cup K\subset P_{m-1}$ and $O\notin T\cup K$. From $x_m\in H_1^-\cap \hat{H}_1^-,x_r\rightarrow x_m$ and 2.4,we have that $x_r\in H_1^-\cap \hat{H}_1^-$. Next, LEMMA D and its proof yield that any $F$ intersects at most two of $\mathcal{V}^w,\mathcal{V}^r$ and $\mathcal{V}^s$, and that $[v_1,u_1,y_1]$ is not a 2-face of $P_m$. Hence , $x_m\in I$ implies that $\lbrace u_1,y_1\rbrace \not\subset I$.\\ \\
\textbf{\uppercase\expandafter{\romannumeral1}.1} $O\notin bd(I)$\\ \  

We apply our Lemmas and Remarks. Then
\begin{itemize}
\item $\lbrace 4\rbrace :x_m\in F: 5.2$
\item $\lbrace 1\rbrace :\mathcal{V}(F)\subset \mathcal{V}(P_{m-1})\cup \mathcal{V}^w: 5.6$
\item $\lbrace 1\rbrace :\mathcal{V}(F)\subset \mathcal{V}(P_{m-1})\cup \mathcal{V}^s: 5.6$
\item $\lbrace 3\rbrace : F\cap \mathcal{V}^w\neq \emptyset \neq F\cap \mathcal{V}^s:E.1$, and 
\item $\lbrace 4\rbrace :\mathcal{V}(F)\subset \mathcal{V}(P_{m})\cup \mathcal{V}^r,F\cap \mathcal{V}^r\neq \emptyset: 5.5$ with $\bar{H}_2,\bar{H}_3,\bar{H}_4,\bar{H}_5$(as $v_1=x_m\notin F$).
\end{itemize}

It remains to consider $F$ that inersect $\mathcal{V}^r$ and $\mathcal{V}^w\cup \mathcal{V}^s$. Here, we apply $\lbrace u_1,y_1\rbrace \not\subset I$ and LEMMA E with relabeling as necessary

\begin{center}
\textbf{\uppercase\expandafter{\romannumeral1}. 1.1} $x_r<x_w<x_s$
\end{center}

As $x_w$ and $x_s$ are interchangeable with respect to $x_r$, we assume that $u_1\notin I$,\\say. Then
\begin{itemize}
\item $\lbrace - \rbrace: F\cap \mathcal{V}^r \neq \emptyset \neq F\cap \mathcal{V}^s: x_r<x_s,u_1\notin I$ and E.2 with $\bar{H}_2,\bar{H}_3$ and $\bar{H}_4$ already counted, and
\item $\lbrace 3 \rbrace: F\cap \mathcal{V}^r\neq \emptyset \neq F\cap \mathcal{V}^w: x_r<x_w,x_r\in H_1^-$ and E.1.
\end{itemize}

\begin{center}
\textbf{\uppercase\expandafter{\romannumeral1}. 1.2} $x_w<x_r<x_s$
\end{center}

If $u_1\notin I$ then one case is above, and 
\begin{itemize}
\item $\lbrace 1 \rbrace: F\cap \mathcal{V}^r \neq \emptyset \neq F\cap \mathcal{V}^w: x_w<x_r,v_1=x_m\notin T$ and E.2 
\end{itemize}

If $u_1\in I$ and $y_1\notin I$ then
\begin{itemize}
\item $\lbrace - \rbrace: F\cap \mathcal{V}^r \neq \emptyset \neq F\cap \mathcal{V}^w: x_w<x_r,x_w\in \bar{H}_1^-$ and E.1, and 
\item $\lbrace 3 \rbrace: F\cap \mathcal{V}^r \neq \emptyset \neq F\cap \mathcal{V}^s: x_r<x_s,x_r\in \hat{H}_1^-$ and E.1 .
\end{itemize}
\begin{center}
\textbf{\uppercase\expandafter{\romannumeral1}. 1.3} $x_w<x_s<x_r$
\end{center}

Then $v_1=x_m\notin T\cup K$ and E.2 yield $\lbrace 1\rbrace$ for $F\cap \mathcal{V}^w \neq \emptyset \neq F\cap \mathcal{V}^r$, and $\lbrace 1 \rbrace$ for $F\cap \mathcal{V}^s \neq \emptyset \neq F\cap \mathcal{V}^r$.\\ \\
\textbf{\uppercase\expandafter{\romannumeral1}.2} $O\in bd(I)$\\ \  

We recall that $\langle x_w,O\rangle \cap T\neq \emptyset$ and $O\notin T$. Hence, $H_1=\langle T\rangle$ strictly separates $O$ and $x_w$, and $x_w$ is necessarily beneath any facet of $P_m$ that contains $O$. Thus $x_w\in \bar{H}_1^-$ and, similarly,$x_s\in \bar{H}_1^-$; whence $\mathcal{V}^w\cup \mathcal{V}^s\subset \bar{H}_1^-$.Since $v_1=x_m$ implies that $\bar{H}_1\cap \bar{H}_5=[v_2,v_i,v_{i+1}]\subset bd(P_{m-1})$, it follows from $O\in \bar{H}_1\backslash P_{m-1}$ that $O\notin [v_2,v_i,v_{i+1}]$. From these observations, we have that 
\begin{itemize}
\item $\lbrace 3 \rbrace :F\cap \mathcal{V}^r = \emptyset:(LEMMA) A$.
\item $\lbrace 8 \rbrace: \mathcal{V}(F)\subset \mathcal{V}(P_m) \cup \mathcal{V}^r,F\cap \mathcal{V}^r\neq \emptyset$ :5.5, A and 5.9 with $\bar{H}_1,\bar{H}_2,\bar{H}_3,\bar{H}_4$ as separating hyperplanes and $O\in \bar{H}_1\cap \bar{H}_j$ for some $j\in \lbrace 2,3,4\rbrace$. We apply LEMMA A and replace $\bar{H}_1$ and $\bar{H}_j$ as per 5.9. We indicate these eight hyperplanes by $2\bar{H}_i+3+3$.
\end{itemize}

We now argue as in \uppercase\expandafter{\romannumeral1}.1.1, \uppercase\expandafter{\romannumeral1}.1.2 and \uppercase\expandafter{\romannumeral1}.3 with 5.9 applied for $\bar{H}_1$ and $\bar{H}_j$, and obtain the same counts. Thus, $s(O)\leq 16$ in each of these cases.

\begin{center}
\textbf{\uppercase\expandafter{\romannumeral2}.} $x_m\in K$ and $x_m\notin T$.
\end{center}

Then $T\subset P_{m-1},O\notin T,x_r\in H_1^-$ and we let $x_m=u_2$. We have again that $\lbrace y_1,u_1\rbrace \not\subset I$; and from $\lbrace v_1,u_1\rbrace \subset K$, it follows that $y_1\notin K$ and $x_w\in \hat{H}_1^-$. We note that $x_m=u_2$ yields that $\hat{H}_4$ separates $O$ from any $F$ with $x_m\notin F$ and $\mathcal{V}(F)\subset \mathcal{V}^s \cup \lbrace u_1,u_2,u_k,u_{k+1} \rbrace$.\\ \\ 
\textbf{\uppercase\expandafter{\romannumeral2}.1} $O\notin bd(I)\cup bd(K)$\\ \ 

Similarly to \uppercase\expandafter{\romannumeral1}.1, we obtain 
\begin{itemize}
\item $\lbrace 4 \rbrace:x_m\in F:5.2$,
\item  $\lbrace 1 \rbrace:\mathcal{V}(F)\subset \mathcal{V}(P_{m-1})\cup \mathcal{V}^s:5.6$,
\item $\lbrace 3 \rbrace: \mathcal{V}(F)\subset \mathcal{V}(P_{m})\cup \mathcal{V}^s,F\cap \mathcal{V}^s\neq \emptyset:5.5$ with $\hat{H}_2,\hat{H}_3,\hat{H}_4$,
\item $\lbrace 4 \rbrace:\mathcal{V}(F)\subset \mathcal{V}(P_{m})\cup \mathcal{V}^r,F\cap \mathcal{V}^r\neq \emptyset:5.5$ with $\bar{H}_2,\bar{H}_3,\bar{H}_4,\bar{H}_5$ and 
\item $\lbrace - \rbrace: F\cap \mathcal{V}^w \neq \emptyset \neq F\cap \mathcal{V}^s: x_w<x_s,x_w\in \hat{H}_1^-$, and E.1. 
\end{itemize}

\begin{center}
\textbf{\uppercase\expandafter{\romannumeral2}.1.1} $x_r<x_w<x_s$
\end{center}

If $u_1\notin I$ then we recall that $x_r\in H_1^-$ and argue as in \uppercase\expandafter{\romannumeral1}.1. If $y_1\notin I$ then
\begin{itemize}
\item $\lbrace - \rbrace: F\cap \mathcal{V}^r \neq \emptyset \neq F\cap \mathcal{V}^w: x_r<x_w$ and E.2.
\end{itemize}
For $F\cap \mathcal{V}^r\neq \emptyset \neq F\cap \mathcal{V}^s$; we obtain from $x_r<x_s$ and $u_1\in I$ that one of E.1, E.3 or E.4 is applicable. We note that E.1 and E.3 yield $\lbrace - \rbrace$, and E.4 yields either $\lbrace - \rbrace$ or $\lbrace 3 \rbrace$ with $O\in I=K$ and 5.9 applied to $\hat{H}_1=\bar{H}_1$.

Henceforth, as a simplification, we list only \lq\lq  worst case scenario\rq\rq results. In that regard, it is noteworthy that the assertion of E.4 is the same if $x_s$ and $x_w$ are interchanged.\\
\begin{center}
\textbf{\uppercase\expandafter{\romannumeral2}.1.2} $x_w<x_r<x_s$ or $x_w<x_s<x_r$
\end{center}

Then as worst case scenarios, we have
\begin{itemize}
\item $\lbrace 1 \rbrace: F\cap \mathcal{V}^w\neq \emptyset \neq F\cap \mathcal{V}^r:x_m=v_1\notin T$ and E.2, and
\item $\lbrace 3 \rbrace: F\cap \mathcal{V}^s\neq \emptyset \neq F\cap \mathcal{V}^r:$ E.4, 5.9 with $O\in \hat{H}_1=\bar{H}_1$.
\end{itemize}
\textbf{\uppercase\expandafter{\romannumeral2}.2} $O\in bd(I)$\\ \  

We note that as in \uppercase\expandafter{\romannumeral1}.2; $x_w\in \bar{H}_1^-$ and
\begin{itemize}
\item $\lbrace 8 \rbrace: \mathcal{V}(F)\subset\mathcal{V}(P_m)\cap \mathcal{V}^r,F\cap\mathcal{V}^r\neq \emptyset: 2\bar{H}_i+3+3$ with $O\in\bar{H}_1\cap\bar{H}_j$ for some $j\in \lbrace 2,3,4\rbrace$.
\end{itemize}

If $x_s\in\bar{H}_1^-$ then $\mathcal{V}^w\cup\mathcal{V}^s\subset\bar{H}_1^-$,
\begin{itemize}
\item $\lbrace 3 \rbrace: F\cap \mathcal{V}^r= \emptyset$: LEMMA A
\end{itemize}
and, as worst case scenario, E.2 and $\lbrace y_1,u_1\rbrace\not\subset I$ yield $x_r<x_w<x_s$ and $u_1\in I $. Then
\begin{itemize}
\item $\lbrace - \rbrace: F\cap \mathcal{V}^r\neq \emptyset \neq F\cap \mathcal{V}^w:y_1\notin I$ and E.2, and
\item $\lbrace 5 \rbrace: F\cap \mathcal{V}^r\neq \emptyset \neq F\cap \mathcal{V}^s:x_r\in \hat{H}_1^-$, E.1, 5.9 with $O\in\hat{H}_i$ for some $i\in \lbrace 2,3,4\rbrace$
\end{itemize}

Let $x_s\in\bar{H}_1^+$. Then $u_1\in I$ with $u_1=v_2$, say, and $\langle x_s,O\rangle\cap\ bd(P_m)\subset I.$ Hence, we choose $K=I$ with $u_2=v_1,u_k=v_i$ and $u_{k+1}=v_{i+1}$. Then $\hat{H}_1=\bar{H}_1$ and
\begin{itemize}
\item $\lbrace 5 \rbrace:\mathcal{V}(F)\subset\mathcal{V}(P_m)\cup\mathcal{V}^s,F\cup\mathcal{V}^s\neq\emptyset:2\hat{H}_i+3$ with $O\in\bar{H}_1\cap\hat{H}_j$ and $\lbrace\hat{H}_j,\hat{H}_i,\hat{H}_i\rbrace=\lbrace \hat{H}_2,\hat{H}_3,\hat{H}_4\rbrace$.
\end{itemize}

From $\lbrace y_1,u_1\rbrace \not\subset I$, we obtain that $y_1\notin I$ and $x_w\in \bar{H}_1^-=\hat{H}_1^-$ and
\begin{itemize}
\item $\lbrace 3\rbrace:\mathcal{V}(F)\subset \mathcal{V}(P_m)\cup \mathcal{V}^w$: LEMMA A,
\item $\lbrace -\rbrace:F\cap \mathcal{V}^w\neq \emptyset \neq F\cap \mathcal{V}^s:x_w<x_s,x_w\in \hat{H}_1^-$ and E.1,
\item $\lbrace -\rbrace:F\cap \mathcal{V}^w\neq \emptyset \neq F\cap \mathcal{V}^r$ : either $x_w<x_r$ and E.1, or $x_r<x_w,y_1\notin I$ and E.2, and
\item $\lbrace -\rbrace:F\cap \mathcal{V}^r\neq \emptyset \neq F\cap \mathcal{V}^s$: E.4 with $\bar{H}_1=\hat{H}_1,\bar{H}_2,\bar{H}_3,\hat{H}_2,\hat{H}_3$.
\end{itemize}
\textbf{\uppercase\expandafter{\romannumeral2}.3} $O\in bd(K)$ and $O\notin bd(I)$\\ \  

We recall that $x_w\in \hat{H}_1^-$, and note that $I\neq K$ implies that $x_r\in \hat{H}_1^-$. Then
\begin{itemize}
\item $\lbrace 3\rbrace:F\cap \mathcal{V}^s=\emptyset$: LEMMA A,
\item $\lbrace 8\rbrace: \mathcal{V}(F)\subset \mathcal{V}(P_m)\cup \mathcal{V}^s,F\cap \mathcal{V}^s\neq \emptyset: 2\hat{H}_i+3+3$ with $O\in \hat{H}_1 \cap \hat{H}_j$ for some $j=\lbrace 2,3,4\rbrace$,
\item $\lbrace -\rbrace:F\cap \mathcal{V}^w\neq \emptyset \neq F\cap \mathcal{V}^s:x_w<x_s,x_w\in \hat{H}_1^-$ and E.1, and as worst case scenario,
\item  $\lbrace 3\rbrace:F\cap \mathcal{V}^s\neq \emptyset \neq F\cap \mathcal{V}^r:x_s<x_r,x_s\in \bar{H}_1^-$ and E.1.
\end{itemize}
\begin{center}
\textbf{\uppercase\expandafter{\romannumeral3}.} $x_m\in T$ and $x_m\notin K$.
\end{center}

As $\lbrace y_1,v_1=x_m\rbrace\subset$\textit{T} and \textit{K}$\subset $\textit{P$_{m-1}$}, we have that $u_1\in H_1^-,\mathcal{V}^s\subset H_1^-,\mathcal{V}^r\subset \hat{H}_1^-$ and $O\notin K$. We let $v_1=y_2$ and note that $H_4$ separates $O$ from any $F$ with $x_m \notin F\subset [\mathcal{V}^w\cup \lbrace y_1,y_t,y_{t+1} \rbrace ]$.\\ \\
\textbf{\uppercase\expandafter{\romannumeral3}.1.} $O\notin bd(I) \cup bd(T)$\\  \ 

As in \uppercase\expandafter{\romannumeral2}.1, we obtain that
\begin{itemize}
\item $\lbrace 4\rbrace :x_m\in F$: 5.2,
\item $\lbrace 1 \rbrace :\mathcal{V}(F)\subset \mathcal{V}(P_{m-1})\cup \mathcal{V}^s$: 5.6,
\item $\lbrace 3\rbrace :\mathcal{V}(F)\subset \mathcal{V}(P_m)\cup \mathcal{V}^w,F\cap \mathcal{V}^w\neq \emptyset$: 5.5 with $H_2,H_3,H_4$,
\item $\lbrace 4\rbrace :\mathcal{V}(F)\subset \mathcal{V}(P_m)\cup \mathcal{V}^r,F\cap \mathcal{V}^r\neq \emptyset$: 5.5 with $\bar{H}_2,\bar{H}_3,\bar{H}_4,\bar{H}_5$, and 
\item $\lbrace -\rbrace :F\cap \mathcal{V}^w\neq \emptyset \neq F\cap \mathcal{V}^s :x_w<x_s,u_1\neq T$ and E.2.
\end{itemize}

We note that our repetitive arguments are dependent upon Lemmas A and E, and $\lbrace u_1,y_1\rbrace \not\subset I$. Also that we present only worst case scenarios.

If $x_r<x_s$ then with $u_1\in I$ and $y_1\notin I$, we have 
\begin{itemize}
\item $\lbrace 3\rbrace :F\cap \mathcal{V}^r\neq \emptyset \neq F \cap \mathcal{V}^s: x_r\in \hat{H}_1^-$ with E.1, and 
\item $\lbrace -\rbrace :F\cap \mathcal{V}^r\neq \emptyset \neq F \cap \mathcal{V}^w$: either $x_r<x_w$ with E.2, or $x_w<x_r$ with E.1.
\end{itemize}

Let $x_w<x_s<x_r$. By E.1, we may assume that $\lbrace x_w,x_s\rbrace \not\subset \bar{H}_1^-$. With $x_s\in \bar
{H}_1^+$ and $x_r\in \hat{H}_1$, we have $u_1\in I, y_1\notin I, x_w\in \bar{H}_1^-$,
\begin{itemize}
\item $\lbrace -\rbrace :F\cap \mathcal{V}^w\neq \emptyset \neq F \cap \mathcal{V}^r$: E.1, and
\item $\lbrace 1\rbrace :F\cap \mathcal{V}^s\neq \emptyset \neq F \cap \mathcal{V}^r$ :E.2 or E.3.\\With $x_w\in \bar{H}_1^+$, we have $x_s\in \bar{H}_1^-$,
\item $\lbrace -\rbrace :F\cap \mathcal{V}^s\neq \emptyset \neq F \cap \mathcal{V}^r$ : E.1, and
\item $\lbrace 3\rbrace :F\cap \mathcal{V}^w\neq \emptyset \neq F \cap \mathcal{V}^r$ : E.4 and 5.9 with $O\in T=I$.
\end{itemize}
\textbf{\uppercase\expandafter{\romannumeral3}.2} $O\in bd(I)$.\\ \  

We note as in \uppercase\expandafter{\romannumeral1}.2 that $x_s\in \bar{H}_1^-$ follows from $O\notin K$. Next, we obtain the same separating hyperplanes for $F$ with $F\cap \mathcal{V}^r\neq \emptyset$ and $\mathcal{V}(F)\subset \mathcal{V}(P_m)\cup \mathcal{V}^r$ as in \uppercase\expandafter{\romannumeral2}.2, and with $x_w$ and $x_s$ interchanged, the corresponding worst case scenario for $x_w\in \bar{H}_1^-$.

Let $x_w\in \bar{H}_1^+$. Then we choose $T=I$ and, similarly to \uppercase\expandafter{\romannumeral2}.2, obtain that 
\begin{itemize}
\item $\lbrace 5\rbrace: \mathcal{V}(P) \subset \mathcal{V}(P_m) \cup \mathcal{V}^w,F\cap \mathcal{V}^w\neq \emptyset: 2H_i+3$ with $H_1=\bar{H_1}$, and $O\in H_1\cap H_j$ for some $j\in \lbrace 2,3,4\rbrace$,
\item $\lbrace 3\rbrace: \mathcal{V}(P) \subset \mathcal{V}(P_m) \cup \mathcal{V}^s$: LEMMA A,
\item $\lbrace -\rbrace :F\cap \mathcal{V}^w\neq \emptyset \neq F \cap \mathcal{V}^s:x_w<x_s,u_1\notin T$ and E.2,
\item $\lbrace -\rbrace :F\cap \mathcal{V}^r\neq \emptyset \neq F \cap \mathcal{V}^s$: either $x_r<x_s,u_1\notin I$ and E.2, or $x_s<x_r,x_s\in \bar{H}_1^-$ and E.1, and
\item $\lbrace -\rbrace :F\cap \mathcal{V}^r\neq \emptyset \neq F \cap \mathcal{V}^s$: E.4.
\end{itemize}
\textbf{\uppercase\expandafter{\romannumeral3}.3} $O\in bd(T)$ and $O\notin bd(I)$.\\ \ 

Then $I\neq T$ and $x_r\in H_1^-$. We recall that $y_1\notin T$and $x_s\in H_1^-$. Hence,
\begin{itemize}
\item $\lbrace 8\rbrace: \mathcal{V}(F)\subset\mathcal{V}(P_m)\cup\mathcal{V}^w,F\cap\mathcal{V}^w\neq \emptyset:2H_i+3+3$ with $O\in H_1\cap H_j$ and $\lbrace H_j,H_i,H_i\rbrace=\lbrace H_2,H_3,H_4\rbrace$.
\item $\lbrace 3\rbrace: F\cap \mathcal{V}^w=\emptyset$: LEMMA A,
\item $\lbrace -\rbrace :F\cap \mathcal{V}^w\neq \emptyset \neq F \cap \mathcal{V}^s:x_w<x_s,u_1\notin T$ and E.2, and as worst case scenario,
\item $\lbrace 3\rbrace :F\cap \mathcal{V}^w\neq \emptyset \neq F \cap \mathcal{V}^r:x_w<x_r,x_w\in \bar{H}_1^-$ and E.1.
\end{itemize}

\begin{center}
\textbf{\uppercase\expandafter{\romannumeral4}} $x_m\in T\cap K$.
\end{center}

We let $v_1=x_m=y_2=u_2$, and note that $\lbrace v_1,y_1\rbrace\subset T$ implies that $u_1\notin T$ and $x_s\in H_1^-$; and $\lbrace v_1,u_1\rbrace \subset K$ implies that $y_1\notin K$ and $x_w\in \hat{H}_1^-$.\\ \\
\textbf{\uppercase\expandafter{\romannumeral4}.1} $O\notin bd(K)\cup bd(T) \cup bd(I)$.\\ \ 

We recall that $O\in P_m \backslash P_{m-1}$. Then 
\begin{itemize}
\item $\lbrace 5\rbrace: x_m\in F$ or $F\subset P_m$: 5.3,
\item $\lbrace 3\rbrace: \mathcal{V}(F)\subset \mathcal{V}(P_m)\cup \mathcal{V}^w\neq \emptyset,y_2=x_m\notin F$: 5.5 with $H_2,H_3$, and $H_4$ for $F\cap Z_t\neq \emptyset$ or $F\cap (Z_t^-\cup Z_t\cup Z_t^+)=\emptyset$,
\item $\lbrace 3\rbrace: \mathcal{V}(F)\subset \mathcal{V}(P_m) \cup \mathcal{V}^s, F\cap \mathcal{V}^s\neq \emptyset, u_2=x_m\notin F$: 5.5 with $\hat{H}_2,\hat{H}_3,\hat{H}_4$,
\item $\lbrace 4\rbrace:\mathcal{V}(F)\subset \mathcal{V}(P_m) \cup \mathcal{V}^r, F\cap \mathcal{V}^r\neq \emptyset $: 5.5 with $\bar{H}_2,\bar{H}_3,\bar{H}_4,\bar{H}_5$, and
\item $\lbrace -\rbrace :F\cap \mathcal{V}^w\neq \emptyset \neq F \cap \mathcal{V}^s:x_w<x_s,x_w\in \hat{H}_1^-$ and E.1.
\end{itemize}

We observe that for $x_r$ and $x_s:$ E.1 and E.2 yield $\lbrace -\rbrace$ for $F\cap \mathcal{V}^r\neq \emptyset \neq F\cap \mathcal{V}^s$, and E.3 and E.4 yield $u_1\in I,v_1\in K$ and the worst case scenario
\begin{itemize}
\item  $\lbrace 1\rbrace :F\cap \mathcal{V}^r\neq \emptyset \neq F \cap \mathcal{V}^s$ : either $x_s<x_r$ with $\hat{H}_5$, or $x_r<x_s$ with $\bar{H}_1=\hat{H}_1$.
\end{itemize}

The corresponding observation for $x_r$ and $x_w$, and $\lbrace u_1,y_1\rbrace \not\subset I$, now yield

\begin{itemize}
\item  $\lbrace 1\rbrace :F\cap \mathcal{V}^r\neq \emptyset \neq F \cap (\mathcal{V}^w\cup \mathcal{V}^s)$. 
\end{itemize}
\textbf{\uppercase\expandafter{\romannumeral4}.2} $O\in bd(K)$.\\ \  

Then $O\notin P_{m-1}$ and $u_2=x_m$ imply that $O\notin [u_1,u_k,u_{k+1}]$ and 
\begin{itemize}
\item $\lbrace 8\rbrace: \mathcal{V}(F) \subset \mathcal{V}(P_m)\cup \mathcal{V}^s, F\cap \mathcal{V}^s\neq \emptyset: 2\hat{H}_i+3+3$ with $O\in \hat{H}_1\cap \hat{H}_j$ and $\lbrace \hat{H}_j,\hat{H}_i,\hat{H}_i\rbrace =\lbrace  \hat{H}_2, \hat{H}_3, \hat{H}_4\rbrace$
\end{itemize}

We recall that $x_w\in \hat{H}_1^-$. If $x_r\in \hat{H}_1^-$ then
\begin{itemize}
\item $\lbrace 3\rbrace:F\cap \mathcal{V}^s=\emptyset$: LEMMA A,
\item $\lbrace -\rbrace:F\cap \mathcal{V}^w\neq \emptyset \neq F\cap \mathcal{V}^s:x_w<x_s$ and E.1, and
\item $\lbrace 3\rbrace:F\cap \mathcal{V}^r\neq \emptyset \neq F\cap \mathcal{V}^s:x_s<x_r$ and either $x_s\in \bar{H}_1^-$ and E.1, or $x_s\in \bar{H}_1^+,v_1\in K,x_r\in \hat{H}_1^-$ and E.3, 5.9 with $O\in \hat{H}_5$.
\end{itemize}

Let $x_r\in \hat{H}_1^+$. Then we choose $I=K$ with $(v_1,v_2,v_i,v_{i+1})=(u_2,u_1,u_k,u_{k+1})$, and note that $y_1\in I = K$ and $x_w \in \bar{H}_1^-=\hat{H}_1^-$. Now 
\begin{itemize}
\item $\lbrace 3\rbrace:\mathcal{V}(F)\subset \mathcal{V}(P_m)\cup \mathcal{V}^w$: LEMMA A with $\hat{H}_1^-$,
\item $\lbrace 5\rbrace:\mathcal{V}(F)\subset \mathcal{V}(P_m)\cup \mathcal{V}^r,F\cap \mathcal{V}^r\neq \emptyset: 2\bar{H}_i+3$ with $O\in \hat{H}_1\cap \bar{H}_j$ and $\lbrace \bar{H}_j,\bar{H}_i,\bar{H}_i \rbrace =\lbrace \bar{H}_2,\bar{H}_3,\bar{H}_4\rbrace$,
\item $\lbrace -\rbrace:F\cap \mathcal{V}^w\neq \emptyset \neq F\cap \mathcal{V}^s:x_w<x_s$ and E.1,
\item $\lbrace -\rbrace:F\cap \mathcal{V}^w\neq \emptyset \neq F\cap \mathcal{V}^r$: either $x_w<x_r$ and E.1, or $x_r<x_w$ and E.2, and
\item $\lbrace -\rbrace:F\cap \mathcal{V}^s\neq \emptyset \neq F\cap \mathcal{V}^r$: E.4.
\end{itemize}
\textbf{\uppercase\expandafter{\romannumeral4}.3} $O\in bd(T)$\\ \  

We argue as in \uppercase\expandafter{\romannumeral4}.2 with
\begin{itemize}
\item $\lbrace 8\rbrace: \mathcal{V}(F)\subset \mathcal{V}(P_m)\cup\mathcal{V}^w,F\cap\mathcal{V}^w\neq\emptyset:2H_i+3+3$,
\item $\lbrace -\rbrace:F\cap \mathcal{V}^w\neq \emptyset \neq F\cap \mathcal{V}^s:x_w<x_s$ and E.2, and the cases $x_r\in H_1^-$ and $x_r\in H_1^+$.
\end{itemize}
\textbf{\uppercase\expandafter{\romannumeral4}.4} $O\in bd(I)$ and $O\notin bd(K)\cup bd(T)$\\ \  

Since we choose $K=I(T=I)$ if $x_s\in\bar{H}_1^+(x_w\in\bar{H}_1^+)$, we may assume that $\lbrace x_w,x_s\rbrace\subset \bar{H}_1^-$. Then
\begin{itemize}
\item $\lbrace 8\rbrace: \mathcal{V}(F)\subset \mathcal{V}(P_m)\cup \mathcal{V}^r, F\cap \mathcal{V}^r\neq \emptyset:2\bar{H}_1+3+3$ with $O\in\bar{H}_1\cap\bar{H}_j$ and $\lbrace \bar{H}_j,\bar{H}_i,\bar{H}_i\rbrace=\lbrace \bar{H}_2,\bar{H}v,\bar{H}_4\rbrace$, and
\item $\lbrace 3\rbrace:F\cap\mathcal{V}^r=\emptyset$: LEMMA A\\ If $u_1\notin I$ ,then
\item $\lbrace -\rbrace:F\cap \mathcal{V}^s\neq \emptyset \neq F\cap \mathcal{V}^r$: either $x_r<x_s$ and E.2, or $x_s<x_r$ and E.1, and
\item $\lbrace 3\rbrace:F\cap \mathcal{V}^w\neq \emptyset \neq F\cap \mathcal{V}^r:x_r<x_w$ and E.1.
\end{itemize}

Let $u_1\in I$. Then $y_1\notin I$ and we argue as above with $x_s$ and $x_w$ interchanged.\qed 

\begin{theorem1}
It remains to determine $s(O)$ in the case that $O\in P_6$ and (in view of 5.1) $O$ is contained in every 4-subpolytope of $P_6$. Since $P_6$ satisfies Gale's Evenness Condition with $x_1<x_2<x_3<x_4<x_5<x_6$ and $O$ is not contained in any facet of $P_6$, it follows that
\begin{eqnarray*}
O\in [x_1,x_3,x_5]\cap [x_2,x_4,x_6].
\end{eqnarray*}
\end{theorem1}

From LEMMA D and its proof, we have 
\begin{itemize}
\item $\mathcal{V}(P)=\lbrace x_1,x_3,x_5\rbrace \cup \mathcal{V}^2\cup\mathcal{V}^4\cup\mathcal{V}^6$,
\item any hyperplane through  $\langle x_1,x_3,x_5\rangle$ intersects at most two of $\mathcal{V}^2,\mathcal{V}^4$ and $\mathcal{V}^6$, and
\item any $F\in \mathcal{F}(P)$ intersects at most two of  $\mathcal{V}^2,\mathcal{V}^4$ and $\mathcal{V}^6$.
\end{itemize}

Let $\mathcal{W}^{ij}=[\lbrace x_1,x_3,x_5\rbrace \cup\mathcal{V}^i\cup\mathcal{V}^j],i\neq j$ in $\lbrace 2,4,6\rbrace$. Then $[x_1,x_3,x_5]\subset bd(\mathcal{W}^{ij})$, any $F\in \mathcal{F}(P)$ is contained is some $\mathcal{W}^{ij}$, and $s(O)\leq 9$ by LEMMA A. \qed   \  

We conclude with the observation that any linked $P$ with $|\mathcal{V}(P)|\leq 11$ is simply linked, and the problem: Is every linked $P$ also simply linked?

\

\noindent University of Calgary\\
2500 University Dr. N.W\\
Calgary, Canada T2N 1N4\\
e-mail: tbisztri@ucalgary.ca
\end{document}